\documentclass[a4paper, 11pt]{amsart}
\usepackage[top=3cm, bottom=3cm, left = 2cm, right = 2cm]{geometry} 
\usepackage[utf8]{inputenc}
\usepackage{textcomp}
\usepackage{graphicx} 
\usepackage{amsmath,amssymb,amsthm}  
\usepackage{bm}  
\usepackage[pdftex,bookmarks,colorlinks,breaklinks]{hyperref}  
\usepackage{memhfixc} 
\usepackage{pdfsync}  

\usepackage[backend=bibtex, style=trad-alpha, doi=true, isbn=false, hyperref=true, backref=false, giveninits=true, abbreviate=true, sorting=nyt]{biblatex}
\hypersetup{colorlinks=true, linkcolor=blue, citecolor=magenta, filecolor=magenta, urlcolor=cyan}

\usepackage{orcidlink}

\newtheorem{theorem}{Theorem}

\newtheorem{lemma}{Lemma}
\newtheorem{obs}[lemma]{Observation}
\newtheorem{remark}{Remark}

\newcommand{\spam}{\mathrm{span}}

\numberwithin{equation}{section}
\numberwithin{lemma}{section}

\begin{document}

\title[Time-periodic solutions to the cubic wave equation]{Time-periodic solutions to the cubic wave equation:\\ an elementary constructive approach}
\author[Filip Ficek]{Filip Ficek\orcidlink{0000-0001-5885-7064}}
\address{University of Vienna, Faculty of Mathematics, Oskar-Morgenstern-Platz 1, 1090 Vienna, Austria}
\address{University of Vienna, Gravitational Physics, Boltzmanngasse 5, 1090 Vienna, Austria}
\email{filip.ficek@univie.ac.at}

\begin{abstract}
We present an elementary proof of existence of infinite family of time-periodic solutions to the one-dimensional nonlinear cubic wave equation with Dirichlet boundary conditions. It relies on the first order perturbative expansion and uses the Banach contraction principle to show existence of nearby solutions. In contrast to the previous results, this approach provides us explicit information about the frequencies and structures of the obtained solutions.
\end{abstract}

\thanks{We thank Maciej Maliborski, Roland Donninger, and Konrad Szyma\'{n}ski for valuable discussions. This work has been supported by the Austrian Science Fund (FWF) through Projects \href{http://doi.org/10.55776/P36455}{P 36455}, \href{http://doi.org/10.55776/PIN2161424}{PIN 2161424}, \href{http://doi.org/10.55776/PAT9429324}{PAT 9429324}, and Wittgenstein Award \href{http://doi.org/10.55776/Z387}{Z 387}.}

\maketitle

\tableofcontents

\section{Introduction}
Time-periodic solutions to the defocusing cubic wave equation with Dirichlet boundary conditions
\begin{align}\label{eq:nlw0}
\partial_t^2 u -\partial_x^2 u + u^3=0, \qquad u(t,0)=u(t,\pi)=0
\end{align}
have been an object of thorough investigations at least since the sixites \cite{keller1966periodic}, see \cite{wayne1997periodic} for an introduction to this topic containing a review of the relevant literature from the twentieth century. 
The current century brought an influx of new works obtained via different approaches, including averaging techniques \cite{Bambusi.2001}, Nash-Moser theorem \cite{Berti.2006, berti2007nonlinear}, variational principle \cite{berti2008cantor}, or Lindstedt series techniques \cite{GM.2004,GMP.2005}. 
They lead to the existence of small solutions with frequencies belonging to a positive measure Cantor set. The recent results suggest, however, that the global structure of the time-periodic solutions to \eqref{eq:nlw0} is much more complicated, containing also large-energy solutions with complex mode composition \cite{FFMM.Nonlinearity,FFMM.ATMP,FFMM.CAP,FFMM.stability}.

The goal of this article is to present a new proof of existence of small solutions to \eqref{eq:nlw0}. It is inspired by the computer assisted proofs \cite{Arioli.2017,FFMM.CAP}, but instead of relying on the numerical approximate solutions, it utilizes the perturbative expansion \cite{Khrustalev.2001, FFMM.Nonlinearity}. The main advantages of this approach are its simplicity (the main tool is the Banach contraction principle) and the explicit information about the frequency and structure of the solutions it provides.

Since \eqref{eq:nlw0} is autonomous, it is practical to introduce the rescaled time $\tau=\Omega t$, so the period becomes equal to $2\pi$ and the considered problem gets the following formulation 
\begin{align}\label{eq:nlw}
\Omega^2\partial_\tau^2 u -\partial_x^2 u+ u^3=0, \qquad u(\tau,0)=u(\tau,\pi)=0,\qquad u(\tau,x)=u(\tau+2\pi,x).
\end{align}
In \cite{FFMM.Nonlinearity} the Poincar\'{e}–Lindstedt method was employed to construct the perturbative expansion to this problem up to the fourth order. It has been proven there, that the lowest order of this expansion in a perturbative parameter $\varepsilon$ can be written as
\begin{align}\label{eq:pert}
u(\tau,x)=\varepsilon \, \sum_{n=0}^{\infty}f_n \sin(2n+1)\tau\,\sin(2n+1)x+\mathcal{O}\left(\varepsilon^3\right), \qquad \Omega=1+ \frac{\varepsilon^2}{256}+\mathcal{O}\left(\varepsilon^4\right),
\end{align}
where
\begin{align*}
f_n=\frac{q^{n+1/2}}{1+q^{2n+1}}
\end{align*}
and $q\approx0.014214$ is the unique solution to
\begin{align}\label{eq:q_equation}
2\left(\sum_{n=0}^\infty q^{(2n+1)^2/4}\right)^4-\left(\frac{1}{2}+\sum_{n=0}^\infty q^{n^2}\right)^4+3\sum_{n=0}^\infty\frac{q^{2n+1}}{(1+q^{2n+1})^2}=0.
\end{align}
Perturbative series for these types of problems are known to be in general divergent \cite{arnold1988dynamical}, except for nowhere dense families of the parameter \cite{GM.2004, GMP.2005}. However, they still can encode a lot of information regarding the structure of the solutions \cite{FFMM.PhysicaD}. 
In this work we use the first order of the perturbative expansion as a satisfactory approximation of the solution to the given problem.

The main obstacle in the study of time-periodic solutions to PDEs is the small-divisors problem. It is usually handled by focusing on solutions with frequencies satisfying certain diophantine conditions \cite{berti2007nonlinear}. Here we take another approach, that proved successful in \cite{Arioli.2017,FFMM.Nonlinearity}, and restrict ourselves to solutions with frequencies having the form $\Omega=\frac{2k+1}{2k}$, where $k\in\mathbb{N}_+$. It gives us a sufficient control on the inverse of the linear part of \eqref{eq:nlw}. The price we have to pay is that the set of the frequencies we cover is of measure zero, however, in contrast to the results employing the diophantine condition, it is explicit.

By restriciting to $\Omega=\frac{2k+1}{2k}$ we can parametrise the first order in \eqref{eq:pert} by $k\in\mathbb{N}_+$ so it becomes
\begin{align*}
    u_k(\tau,x)=\sqrt{\frac{128}{k}} \sum_{n=0}^{\infty}f_n \sin(2n+1)\tau\,\sin(2n+1)x.
\end{align*}
This function will play the role of an approximate solution to \eqref{eq:nlw} for a fixed $k$. We show that for sufficiently large values of $k$ there exists an exact solution to \eqref{eq:nlw} in a small neighbourhood of this approximate.

\begin{theorem}\label{thm:main}
For every $k\geq 79\,675$ there exists a non-trivial solution $u$ to \eqref{eq:nlw} with $\Omega=\frac{2k+1}{2k}$. This solution satisfies $\Vert u-u_k\Vert<\frac{139}{42\,500} k^{-1/2}$, with the norm $\Vert\cdot\Vert$ defined below.
\end{theorem}

\begin{remark}
We focus here on the solutions bifurcating from $\Omega=1$ and dominated by the mode $\sin\tau\,\sin x$. Other families can be obtained by simple recalings
\begin{align*}
    \tilde{u}(\tau,x)=n\,u(m\tau,nx),
\end{align*}
where $m,n\in\mathbb{N}_+$. The frequency of such rescaled solution $\tilde{u}$ is $n\Omega/m$, where $\Omega$ is the frequency of $u$.
\end{remark}

\begin{remark}
The constructed solutions $u$ to \eqref{eq:nlw0} have the property that by defining
\begin{align*}
    \tilde{u}(\tau,x)=\Omega^{-1}u(x,\tau)
\end{align*}
we get functions that satisfy
\begin{align*}
\Omega^{-2}\partial_\tau^2 \tilde{u} -\partial_x^2 \tilde{u}- \tilde{u}^3=0, \qquad \tilde{u}(\tau,0)=\tilde{u}(\tau,\pi)=0,\qquad \tilde{u}(\tau,x)=\tilde{u}(\tau+2\pi,x).
\end{align*}
Hence, Theorem \ref{thm:main} automatically gives existence of an infinite family of time-periodic solutions to the focusing cubic wave equation with frequencies having the form $\frac{2k}{2k+1}$.
\end{remark}

We present the proof of Theorem \ref{thm:main} in Section \ref{sec:proof}. It utilizes a few bounds on norms of various vectors and operators. Sections \ref{sec:preliminaries}-\ref{sec:boundH} are devoted to proofs of those bounds.

\section{Strategy of the proof}\label{sec:proof}
\subsection{Functional setting}
Let us define a family of basis functions $P_{m,n}$
\begin{equation}
    \label{eq:Pmn}
    P_{m,n}(\tau,x):=\sin (2m+1)\tau\,\sin(2n+1)x,
\end{equation}
where $m$ and $n$ are non-negative integers.
Finite linear combinations of functions $P_{m,n}$ constitute the vector space $\tilde{X}$. For any $v\in\tilde{X}$ we have a unique decomposition into
\begin{align*}
v=\sum_{m,n} \hat{v}_{m,n}P_{m,n},
\end{align*}
where the sum is finite. Then we equip $\tilde{X}$ with the weighted $l^1$-norm defined by
\begin{align*}
\Vert v\Vert=\sum_{m,n} \rho^{2(m+n+1)}  \left|\hat{v}_{m,n}\right|,
\end{align*}
where the weights are fixed to be $\rho=1+\frac{1}{1000}$. The completion of $\tilde{X}$ with respect to the norm $\Vert\cdot\Vert$ constitutes the Banach space $X$ in which we will be working. Such a construction of $X$ has the following advantages:
\begin{itemize}
    \item All elements of $X$ are $2\pi$-periodic in $\tau$ and are identically zero at $x=0$ and $x=\pi$. Thus, they automatically satisfy the boundary conditions imposed in \eqref{eq:nlw}.
    \item Thanks to the weights $\rho>1$, the elements of $X$ are analytic inside the strip \cite{arnold1988dynamical}
    \begin{align*}
        \left\{(\tau,x)\in\mathbb{C}^2:|\Im\tau|<\rho-1,|\Im x|<\rho-1\right\}.
    \end{align*}
    Hence, they are double differentiable in real variables.
    \item The $l^1$ norm lets us control the nonlinear term since, as we show in Lemma \ref{lem:triple_norm}, for any $u,v,w\in X$ their pointwise product $uvw$ belongs to $X$ with $\Vert uvw\Vert\leq \Vert u\Vert \Vert v\Vert \Vert w\Vert$.
    \item The $l^1$-norm allows for simple estimations of operator norms, since for any $\mathcal{H}\in\mathcal{L}(X)$ it holds
    \begin{align*}
        \Vert \mathcal{H}\Vert\leq \sup_{m,n}\frac{\left\Vert \mathcal{H}P_{m,n}\right\Vert}{\Vert P_{m,n}\Vert}.
    \end{align*}
\end{itemize}

We end this subsection by introducing additional notation and conventions that will be used throughout the paper. An open ball centered in $u\in X$ with radius $\delta\geq 0$ is denoted by $B_{\delta}(u)$. We define two subspaces of $X$
\begin{align*}
    Y_1=\spam\left\{P_{0,0},P_{0,1},P_{1,0}\right\},\qquad Y_2=\spam\{P_{m,n}:m\leq3\mbox{ and }n\leq3\},
\end{align*}
together with their respective complements $Z_1$ and $Z_2$, so that we have two different decompositions of $X$, namely $X=Y_1\oplus Z_1$ and $X=Y_2\oplus Z_2$.
Finally, certain formulas used below can be more conveniently expressed with the use of negative indices in $P_{m,n}$. Hence, we naturally extend the definition \eqref{eq:Pmn} so that
\begin{align}\label{eq:negative}
    P_{-m,n}=-P_{m-1,n}, \qquad P_{m,-n}=-P_{m,n-1}, \qquad P_{-m,-n}=P_{m-1,n-1}.
\end{align}

\subsection{Fixed-point formulation}
We rewrite the problem \eqref{eq:nlw} as a fixed-point problem. To do so, let us define the differential operator 
\begin{align*}
    L_k=\left(\frac{2k+1}{2k}\right)^2\partial_\tau^2-\partial_x^2.
\end{align*}
As we show in Observation \ref{obs:L-1norm}, it has a well-defined, bounded inverse $L_k^{-1}:X\to X$. Now we can decompose the sought solution as $u=u_k+Ah$, where $A:X\to X$ is a linear, continuous isomorphism defined in Section \ref{sec:A}. Then, \eqref{eq:nlw} can be written as a fixed-point problem $\mathcal{N}_k(h)=h$ for $h\in X$, where
\begin{align*}
    \mathcal{N}_k(h)=-L_k^{-1}\left(u_k+A h\right)^3-u_k+(I-A)h.
\end{align*}
Let us denote the linear part of $\mathcal{N}_k$ by
\begin{align*}
H_k(h)=-3L_k^{-1}(u_k^2\,Ah)+h-Ah.
\end{align*}
This formulation lets us give a simple proof of Theorem \ref{thm:main} relying on a few bounds that are derived in the remainder of the paper. Using these bounds we are able to show that for sufficiently large $k$ the nonlinear operator $\mathcal{N}_k$ is a contraction inside some sufficiently small neighborhood of zero, hence, it has a fixed point that gives rise to a solution of \eqref{eq:nlw} close to $u_k$. 

\begin{proof}[Proof of Theorem 1]
    Let us fix any $k\geq 79\,675$. We invoke the following bounds from the succeeding sections
\begin{align*}
    \Vert L_k^{-1}\Vert &< \frac{4k^2}{4k-1},&\quad \mbox{(Observation \ref{obs:L-1norm})},\\  
    \Vert u_k \Vert &<\frac{3}{2}k^{-1/2},&\quad \mbox{(Observation \ref{obs:bounduk})},\\
    \Vert \mathcal{N}_k(0)\Vert &< 8\sqrt{2} k^{-3/2},&\quad \mbox{(Lemma \ref{lem:boundN0})},\\
    \Vert A \Vert &<\frac{139}{85},&\quad \mbox{(Lemma \ref{lem:boundA})},\\
    \Vert H_k \Vert &<\frac{88}{100},&\quad \mbox{(Lemma \ref{lem:boundHk})}.
\end{align*}
Let us introduce $\delta_k=\frac{1}{500}k^{-1/2}$, so the bounds listed above yield
\begin{align*}
    \Vert H_k\Vert+6\left\Vert L_k^{-1}\right\Vert \Vert u_k\Vert \Vert A \Vert^2 \delta_k+3\left\Vert L_k^{-1}\right\Vert \Vert A \Vert^3 \delta_k^2 <\frac{88}{100}+\frac{7\,398\,339\,357}{38\,382\,812\,500}\frac{k}{4k-1}<\frac{929}{1000},
\end{align*}
where the last inequality uses the fact that $k$ is sufficiently large. Then, for any $h_1,h_2\in B_{\delta_k}(0)$ we get
\begin{align*}
\left\Vert\mathcal{N}_k(h_1)-\mathcal{N}_k(h_2)\right\Vert\leq&\left(\Vert H_k\Vert+6\left\Vert L_k^{-1}\right\Vert \Vert A\Vert^2 \Vert u_k\Vert\delta_k +3\left\Vert L_k^{-1}\right\Vert \Vert A\Vert^3\delta_k^2\right)\Vert h_1-h_2 \Vert<\frac{929}{1000}\Vert h_1-h_2 \Vert.
\end{align*}
This inequality also implies that for $h\in B_{\delta_k}$ it holds
\begin{align*}
\left\Vert\mathcal{N}_k(h)\right\Vert&\leq\left\Vert\mathcal{N}_k(h)-\mathcal{N}_k(0)\right\Vert+\left\Vert\mathcal{N}_k(0)\right\Vert  < \frac{929}{1000}\delta_k + 8\sqrt{2} k^{-3/2}=\left(\frac{929}{1000}+4000\sqrt{2}k^{-1}\right)\delta_k<\delta_k,
\end{align*}
because $k\geq79\,675$.
As a result, $\mathcal{N}_k:B_{\delta_k}\to B_{\delta_k}$ is a contraction. Hence, via the Banach contraction principle it has a unique fixed point $h\in B_{\delta_k}$. It leads to the solution $u=u_k+Ah$ of \eqref{eq:nlw} with $\Omega=\frac{2k+1}{2k}$. Additionally, we have
\begin{align*}
\Vert u_k-u\Vert<\Vert A\Vert\delta_k<\frac{139}{42\,500} k^{-1/2}.
\end{align*}
Since from Observation \ref{obs:bounduk} we have $\Vert u_k\Vert>\frac{5}{4}k^{-1/2}$, it holds
\begin{align*}
\Vert u\Vert>\Vert u_k\Vert-\Vert u_k-u\Vert>\left(\frac{5}{4}-\frac{139}{42\,500}\right) k^{-1/2}>0
\end{align*}
and the solution $u$ is indeed non-trivial.
\end{proof}

\section{Preliminaries}\label{sec:preliminaries}
In this section we collect some basic results that are used throughout the paper. For the convenience, we divide it into three short subsections.
\subsection[Space X]{Space $X$}
We begin by showing that $X$ is closed under the pointwise multiplication of its three elements. In particular, it means that for any $u\in X$, also $u^3\in X$.
\begin{lemma}\label{lem:triple_norm}
    Let $u,v,w\in X$, then
    \begin{align*}
        \Vert u\, v\, w\Vert\leq \Vert u\Vert\,\Vert v\Vert\,\Vert w \Vert.
    \end{align*}
\end{lemma}
\begin{proof}
    The decomposition of a triple product of $P_{m,n}$ leads to 
    \begin{align}\label{eq:PPP}
        P_{m_1,n_1}\,P_{m_2,n_2}\,P_{m_3,n_3}=\frac{1}{16}\left[\right.& P_{m_1+m_2+m_3+1,n_1+n_2+n_3+1}-P_{-m_1+m_2+m_3,n_1+n_2+n_3+1}\\
        &-P_{m_1-m_2+m_3,n_1+n_2+n_3+1}- P_{m_1+m_2-m_3,n_1+n_2+n_3+1}\nonumber\\
        &-P_{m_1+m_2+m_3+1,-n_1+n_2+n_3}+P_{-m_1+m_2+m_3,-n_1+n_2+n_3}\nonumber\\
        &+P_{m_1-m_2+m_3,-n_1+n_2+n_3}+ P_{m_1+m_2-m_3,-n_1+n_2+n_3}\nonumber\\
        &-P_{m_1+m_2+m_3+1,n_1-n_2+n_3}+P_{-m_1+m_2+m_3,n_1-n_2+n_3}\nonumber\\
        &+P_{m_1-m_2+m_3,n_1-n_2+n_3}+ P_{m_1+m_2-m_3,n_1-n_2+n_3}\nonumber\\
        &-P_{m_1+m_2+m_3+1,n_1+n_2-n_3}+P_{-m_1+m_2+m_3,n_1+n_2-n_3}\nonumber\\
        &\left.+P_{m_1-m_2+m_3,n_1+n_2-n_3}+ P_{m_1+m_2-m_3,n_1+n_2-n_3}\right],\nonumber
    \end{align}
where potential negative indices are treated as in \eqref{eq:negative}. Thus, it holds
\begin{align*}
        \left\Vert P_{m_1,n_1}\,P_{m_2,n_2}\,P_{m_3,n_3}\right\Vert \leq \rho^{2(m_1+m_2+m_3+n_1+n_2+n_3+3)}.
\end{align*} 

Now assume that $u,v,w\in \tilde{X}$, so we can write them as
    \begin{align*}
        u=\sum_{m,n} \hat{u}_{m,n}P_{m,n},\qquad 
        v=\sum_{m,n} \hat{v}_{m,n}P_{m,n},\quad 
        w=\sum_{m,n} \hat{w}_{m,n}P_{m,n}.
    \end{align*}
Then
    \begin{align*}
        \Vert u\,v\,w\Vert&\leq 
        \sum_{m_1,n_1}
        \sum_{m_2,n_2} 
        \sum_{m_3,n_3} 
        \left|\hat{u}_{m_1,n_1} \hat{v}_{m_2,n_2} \hat{w}_{m_3,n_3}\right| \left\Vert P_{m_1,n_1}P_{m_2,n_2} P_{m_3,n_3}\right\Vert\\
        &\leq
        \sum_{m_1,n_1}
        \sum_{m_2,n_2} 
        \sum_{m_3,n_3}  \left|\hat{u}_{m_1,n_1} \hat{v}_{m_2,n_2} \hat{w}_{m_3,n_3}\right|\rho^{2(m_1+m_2+m_3+n_1+n_2+n_3+3)} \\
        &=\left(\sum_{m_1,n_1}\rho^{2(m_1+n_1+1)} \left|\hat{u}_{m_1,n_1}\right| \right)
        \left(\sum_{m_2,n_2} \rho^{2(m_2+n_2+1)}\left|\hat{v}_{m_2,n_2}\right| \right) \\
        &\quad \times\left(\sum_{m_3,n_3} \rho^{2(m_3+n_3+1)}\left|\hat{u}_{m_1,n_1}\right| \right)
        =\Vert u\Vert\,\Vert v\Vert\,\Vert w \Vert.
    \end{align*}
    As $X$ is the completion of $\tilde{X}$, this concludes the proof.
\end{proof}

\subsection[Operator L-1]{Operator $L_k^{-1}$}
The operator $L_k$ acts on the basis functions as
\begin{align*}
L_k P_{m,n}=\left[-\frac{(2k+1)^2}{4k^2}(2m+1)^2+(2n+1)^2\right]P_{m,n}.
\end{align*}
Since the expression inside the square brackets is non-zero for every $(m,n)\in\mathbb{N}^2$, we can formally define $L_k^{-1}$ as
\begin{align*}
L_k^{-1}P_{m,n}=\frac{4k^2}{-(2k+1)^2(2m+1)^2+4k^2(2n+1)^2}P_{m,n}.
\end{align*}
Now we prove a bound on $L^{-1}_k$ that implies continuity of this operator.
We begin with the following simple inequality
\begin{lemma}\label{lem:inequality}
Let $t$ and $s$ be positive numbers such that $|t-s|\geq 1$. Then the following inequality holds
\begin{equation}
    \label{eq:ineq}
    |t^2-s^2|\geq 2\max\{t,s\}-1.
\end{equation}
\end{lemma}
\begin{proof}
Assume that $s<t$ and consider a function $f(s)=t^2-s^2$ defined on an interval $s\in(0,t-1]$ for some fixed $t$. The function $f$ achieves its minimum at $s=t-1$, hence $(t^2-s^2)\geq 2t-1$. The case when $s>t$ is analogous, leading to \eqref{eq:ineq}.
\end{proof}
\noindent Now we are ready to prove
\begin{lemma}\label{lem:L-1ineq}
For any $(m,n)\in\mathbb{N}^2$ it holds
    \begin{align}\label{eq:L-1ineq}
        \left\Vert L^{-1} P_{m,n}\right\Vert\leq\frac{4k^2}{2\max\left\{2k(2n+1),(2k+1)(2m+1) \right\}-1}\left\Vert P_{m,n}\right\Vert
    \end{align}    
\end{lemma}
\begin{proof}
Since $(2k+1)(2m+1)$ is odd and $2k(2n+1)$ is even, the difference between these two numbers is no smaller than one. Then from \eqref{eq:ineq} we get
\begin{align*}
    \left|-(2k+1)^2(2m+1)^2+4k^2(2n+1)^2\right|\geq 2\max\{2k(2n+1),(2k+1)(2m+1)\} - 1.
\end{align*}
As a result
\begin{align*}
\left\Vert L_k^{-1}(P_{m,n})\right\Vert&=\left|\frac{4k^2}{-(2k+1)^2(2m+1)^2+4k^2(2n+1)^2}\right|\left\Vert P_{m,n}\right\Vert\\
&\leq \frac{4k^2}{2\max\left\{2k(2n+1),(2k+1)(2m+1)\right\}-1}\left\Vert P_{m,n}\right\Vert.
\end{align*}
\end{proof}
\noindent This lemma leads to the following observations.
\begin{obs}\label{obs:L-1norm}
The inverse operator $L_k^{-1}$ is bounded by
\begin{align*}
    \left\Vert L_k^{-1}\right\Vert\leq \frac{4k^2}{4k-1}.
\end{align*}
\end{obs}
\begin{obs}\label{obs:L-1negative}
For any $(\mu,\nu)\in\mathbb{Z}^2$ it holds
    \begin{align}\label{eq:L-1negative}
        \left\Vert L^{-1} P_{\mu,\nu}\right\Vert\leq\frac{4k^2}{2\max\left\{2k|2\nu+1|,(2k+1)|2\mu+1| \right\}-1}\left\Vert P_{\mu,\nu}\right\Vert.
    \end{align} 
We can also get a simpler but weaker bound
\begin{align}\label{eq:L-1negative2}
     \left\Vert L^{-1} P_{\mu,\nu}\right\Vert\leq\frac{4k^2}{4k\max\left\{|2\mu+1|,|2\nu+1| \right\}-1}\left\Vert P_{\mu,\nu}\right\Vert.
\end{align}
\end{obs}

\subsection[Function u]{Function $u_k$}
Here we derive rigorous bounds on the value of $q$ present in $f_n$.
\begin{lemma}\label{lem:q}
    For the solution to \eqref{eq:q_equation} it holds $q\in(\frac{13}{1000},\frac{15}{1000})$.
\end{lemma}
\begin{proof}
Let us define the following function on an interval $[0,1)$, see \eqref{eq:q_equation},
\begin{align*}
g(x)=2\left(\sum_{n=0}^\infty x^{(2n+1)^2/4}\right)^4-\left(\frac{1}{2}+\sum_{n=0}^\infty x^{n^2}\right)^4+3\sum_{n=0}^\infty\frac{x^{2n+1}}{(1+x^{2n+1})^2},
\end{align*}
together with auxiliary functions
\begin{align*}
\underline{g}(x)=&2x-\left(\frac{1}{2}+x+\frac{x^4}{1-x}\right)^4+\frac{3x}{(1+x)^2},\\
\overline{g}(x)=& \frac{2x}{(1-x^2)^4}-\left(\frac{1}{2}+x\right)^4+\frac{3x}{1-x^2}
\end{align*}
Then it holds $\underline{g}(x)\leq g(x)\leq\overline{g}(x)$. A simple evaluation shows that $\overline{g}(\frac{13}{1000})<0$ and $\underline{g}(\frac{15}{1000})>0$. Since from \cite{FFMM.Nonlinearity} we know that $g$ has a unique root, it must be inside the interval $(\frac{13}{1000},\frac{15}{1000})$.
\end{proof}
\noindent This result lets us get explicit lower and upper bounds on various expressions encountered throughout the paper. As an example, we get the following estimates for $\Vert u_k \Vert$.
\begin{obs}\label{obs:bounduk}
    Since $0<q<1$ and $f_n<q^{n+1/2}$, we have $u_k \in X$. Additionally, Lemma \ref{lem:q} gives us
\begin{align*}
    \left\Vert u_k\right\Vert <\sqrt{\frac{128}{k}}\frac{\rho^2\sqrt{q}}{1-\rho^4 q}<\frac{3}{2\sqrt{k}}
\end{align*}
and
\begin{align*}
    \left\Vert u_k\right\Vert >\sqrt{\frac{128}{k}}\frac{\rho^2\sqrt{q}}{1+q}>\frac{5}{4\sqrt{k}}.
\end{align*}
\end{obs}

\section[Bound on the residue N(0)]{Bound on the residue $\mathcal{N}_k(0)$}
In this section we derive the bound on the norm of
\begin{align*}
    \mathcal{N}_k(0)=-L_k^{-1}u_k^3-u_k.
\end{align*}
We start by investigating the cubic power of $u_k$. Formulas derived in \cite{FFMM.Nonlinearity} let one write it explicitly as
\begin{align*}
u_k^3 (\tau,x)=\frac{1024\sqrt{2}}{k^{3/2}}\sum_{n=0}^{\infty} \sum_{m=0}^{\infty} b_{m,n} \,P_{m,n}(\tau,x),
\end{align*}
where
\begin{align*}
b_{m,n}=\begin{cases}
\displaystyle\frac{(2m+1)^2}{128}f_m \quad & \mbox{ for $m=n$},\\[3ex]
\displaystyle-\frac{3}{64}\frac{m-n}{\sinh\left[\left(m+\frac{1}{2}\right)\ln q\right] - \sinh\left[\left(n+\frac{1}{2}\right)\ln q\right]} & \mbox{ for $m\neq n$ and $m-n$ even},\\[3ex]
\displaystyle\frac{3}{64}\frac{m+n+1}{\sinh\left[\left(m+\frac{1}{2}\right)\ln q\right] + \sinh\left[\left(n+\frac{1}{2}\right)\ln q\right]} & \mbox{ for $m\neq n$ and $m-n$ odd}.
\end{cases}
\end{align*}
This expression leads to the following bound on $\left\Vert u_k^3\right\Vert$.
\begin{lemma}\label{lem:bounduk3}
    It holds $\left\Vert u_k^3\right\Vert<2\sqrt{2}k^{-3/2}$.
\end{lemma}
\begin{proof}
Let us consider
\begin{align*}
    I:= \sum_{m=0}^\infty  \sum_{n=0}^\infty \left|b_{m,n}\right|\rho^{2(m+n+1)}.
\end{align*}
We separate it into two parts, diagonal ($m=n$) and nondiagonal ($m\neq n$). The former satisfies
\begin{align*}
    \sum_{m=0}^\infty \left|b_{m,m}\right|\rho^{4m+2}&=\sum_{m=0}^\infty\frac{(2m+1)^2}{128}f_m\rho^{4m+2}< \sum_{m=0}^\infty\frac{(2m+1)^2}{128}q^{m+1/2}\rho^{4m+2}\\
    &=\frac{\sqrt{\rho^4 q}}{128}\frac{1+6\rho^4 q+\rho^8 q^2}{(1-\rho^4 q)^3}<\frac{11}{10\,000}.
\end{align*}

To bound the non-diagonal terms we assume that $m>n$ and consider two cases. First, let $m-n$ be odd, so
\begin{align*}
    \sinh\left[\left(m+\frac{1}{2}\right)\ln q\right] + \sinh\left[\left(n+\frac{1}{2}\right)\ln q\right]=-\frac{1}{2q^{m+1/2}}\left(1+q^{m-n}-q^{m+n+1}-q^{2m+1}\right).
\end{align*}
Since $1-q^{2n+1}-q^{m+n+1}>1-2q>0$, the expression inside the brackets is larger than one. Thus, we have
\begin{align}\label{eq:u3_proof_A}
    |b_{m,n}|<\frac{3}{32} q^{m+1/2}(m+n+1).
\end{align}
If $m-n$ is even, we need to consider
\begin{align*}
    \sinh\left[\left(m+\frac{1}{2}\right)\ln q\right] - \sinh\left[\left(n+\frac{1}{2}\right)\ln q\right]=-\frac{1}{2q^{m+1/2}}\left(1-q^{m-n}+q^{m+n+1}-q^{2m+1}\right).
\end{align*}
This time we can easily bound the bracket to be larger than $1-2q$, getting
\begin{align}\label{eq:u3_proof_B}
    |b_{m,n}|<\frac{3}{32} q^{m+1/2}\frac{m-n}{1-2q}.
\end{align}
We combine \eqref{eq:u3_proof_A} and \eqref{eq:u3_proof_B} into a common bound for $m>n$
\begin{align*}
    |b_{m,n}|<\frac{3}{32} q^{m+1/2}\frac{m+n+1}{1-2q}.
\end{align*}
An analogous bound holds for $m<n$, thus, we can write
\begin{align*}
    \sum_{m\neq n} |b_{m,n}|\rho^{2(m+n+1)}&<2\sum_{m=1}^{\infty}\sum_{n=0}^{m-1}\frac{3}{32} q^{m+1/2}\frac{m+n+1}{1-2q} \rho^{2(m+n+1)}\\
    &=\frac{3\rho^4q^{3/2}(2-\rho^2 q-\rho^6 q)}{16(1-2q)(1-\rho^2 q)^2(1-\rho^4 q)^2}<\frac{8}{10\,000},
\end{align*}
where the sums has been computed using standard summation identities
\begin{align*}
\sum_{n=0}^{m-1} p^n = \frac{1-p^m}{1-p},\qquad \sum_{n=0}^{m-1} n\,p^n = \frac{p\,(1-p^m)}{(1-p)^2} - \frac{m\,p^m}{1-p}.
\end{align*}
Combining this result with the bound on the sum of the diagonal terms we get $I<\frac{19}{10\,000}<\frac{1}{512}$. This leads to $\Vert u_k^3\Vert<2\sqrt{2}k^{-3/2}$
\end{proof}

We also prove the following lemma.
\begin{lemma}\label{lem:bound_help_N0}
For any $k\in\mathbb{N}_+$ and $(m,n)\in\mathbb{N}^2$ such that $m<n$ it holds 
\begin{align}\label{eq:N0_proof_inequality}
\left|\frac{k^2}{16(m-n)(m+n+1)k^2+(2m+1)^2 (4k+1)}\right|\leq 1.
\end{align}
\end{lemma}
\begin{proof}
Let us consider 
\begin{align*}
    g(x)=\frac{x^2}{-\alpha\, x^2+\beta\,(4x+1)},
\end{align*}
where $\alpha,\beta>0$ and $x\geq 0$. It has a pole at $x_0=2\beta(1+\sqrt{1+\alpha/2\beta})/\alpha$ and is increasing in the intervals $(0,x_0)$, $(x_0,\infty)$. In addition, its values are positive in the former and negative in the latter. For any $m, n\in\mathbb{N}$ such that $m<n$ we can put $\alpha=16(n-m)(m+n+1)$, $\beta=(2m+1)^2$, so that $|g(k)|$ becomes the left hand side of \eqref{eq:N0_proof_inequality}. Then $x_0=(2m+1)/(4(n-m))$. We are interested in the maximal value of $|g(k)|$ for $k\in\mathbb{N}_+$. From the considerations above, it follows that $\max_{k\in\mathbb{N}}|g(k)|=\max(|g(\lfloor x_0\rfloor)|,|g(\lceil x_0\rceil)|)$. We can use now \eqref{eq:L-1ineq} to bound
\begin{align*}
    |g(k)|=\left|\frac{k^2}{-(2k+1)^2(2m+1)^2+4k^2(2n+1)^2}\right|\leq\frac{k^2}{2\max\{2k(2n+1),(2k+1)(2m+1)\}-1}.
\end{align*}

Since $\lceil x_0\rceil\in[x_0,x_0+1)$, we get
\begin{align*}
    |g(\lceil x_0\rceil)|\leq\frac{\left(x_0+1\right)^2}{2\max\{2x_0(2n+1),(2x_0+1)(2m+1)\}-1}
    =\frac{(1-2m+4n)^2}{16(n-m)(1+3m+n+4mn)}.
\end{align*}
Since $m<n$, we can introduce $l=n-m$ and define
\begin{align*}
\overline{g}(m,l)=\frac{(1+2m+4l)^2}{16l[(4m+1)l+(2m+1)^2]}.
\end{align*}
For any fixed $l\geq 1$, $\overline{g}$ is a decreasing function in $m$. Hence, we can look for its largest value on the line $m=0$. For $l\geq 1$ the function $\overline{g}(0,l)$ is increasing. Since $\lim_{l\to\infty}\overline{g}(0,l)=1$, we get a bound $\overline{g}(m,n)\leq 1$ leading to $|g(\lceil x_0\rceil)|\leq1$.

The other argument $\lfloor x_0\rfloor$ requires consideration of two subcases. When $x_0<1$, then $g(\lfloor x_0\rfloor)=g(0)=0$. Hence, let us assume $x_0>1$, which is equivalent to $1+6m-4n>0$, so $\lfloor x_0\rfloor\in(x_0-1,x_0]$. Similarly as before, we then have
\begin{align*}
    |g(\lfloor x_0\rfloor)|&\leq\frac{x_0^2}{2\max(2(x_0-1)(2n+1),(2x_0-1)(2m+1))-1}
    \\
    &=\frac{(2m+1)^2}{16(n-m)[(2m+1)(1+6m-4n)+(1+4m)(n-m)]}.
\end{align*}
Let us again introduce $l=n-m$ and define
\begin{align*}
\underline{g}(m,l)=\frac{(2m+1)^2}{16l[(2m+1)^2-(4m+3)l]}.
\end{align*}
We are interested in the values of this function inside the region $l\geq1$, $m>(1+4l)/2$. The second condition is equivalent to $x_0>1$ and ensures positivity of $\underline{g}$. For fixed $l$, $\underline{g}$ is a decreasing function in $m$, so the maximal value on that line can be bounded by $\underline{g}((1+4l)/2,l)$. This expression is decreasing as a function of $l$, leading us to the bound $\underline{g}(5/2,1)=9/92$. As a result, we get $|g(\lfloor x_0\rfloor)|\leq9/92$. Together with the bound on $|g(\lceil x_0\rceil)|$ it gives \eqref{eq:N0_proof_inequality}.

\end{proof}

Now we are ready to derive the estimate for $\Vert\mathcal{N}_k(0)\Vert$.
\begin{lemma}\label{lem:boundN0}
It holds $\Vert \mathcal{N}_k(0)\Vert \leq 8\sqrt{2}k^{-3/2}$.
\end{lemma}
\begin{proof}
Let us fix $k\in\mathbb{N}_+$ and denote the coefficients in the decomposition of $ \mathcal{N}_k(0)$ in the $P_{m,n}$ basis by $\hat{N}_{m,n}$, so 
that we have
\begin{align*}
    \mathcal{N}_k(0)=\sum_{m=0}^\infty \sum_{n=0}^\infty \hat{N}_{m,n} P_{m,n}.
\end{align*}
Again we analyse diagonal ($m=n$) and non-diagonal ($m\neq n$) modes separately. For the former it holds
\begin{align*}
\left|\hat{N}_{m,m}\right|&=\left|-\frac{4k^2}{(4k^2-(2k+1)^2)(2m+1)^2} \frac{1024\sqrt{2}}{k\sqrt{k}}\,b_{m,m}-\frac{16}{\sqrt{2k}}f_m\right|=\frac{1024\sqrt{2}}{(4k+1)\sqrt{k}(2m+1)^2}\,\left|b_{m,m}\right|\\
&\leq 256\sqrt{2}\,k^{-3/2}\,\left|b_{m,m}\right|,
\end{align*}
while for the latter
\begin{align*}
\left|\hat{N}_{m,n}\right|&=\left|\frac{4k^2}{-(2k+1)^2(2m+1)^2+4k^2(2n+1)^2 } \frac{1024\sqrt{2}}{k\sqrt{k}}\, b_{m,n}\right|\\
&=\left|\frac{4096 \sqrt{2k}}{16(m-n)(m+n+1)k^2+(2m+1)^2 (4k+1)}\, b_{m,n}\right|.
\end{align*}
Now we consider two cases. When $m>n$, we get
\begin{align*}
\left|\hat{N}_{m,n}\right|&\leq\frac{4096 \sqrt{2k}}{16(m-n)(m+n+1)k^2} \left|b_{m,n}\right|
\leq256\sqrt{2}k^{-3/2}\,\left|b_{m,n}\right|.
\end{align*}
In the opposite case we can use \eqref{eq:N0_proof_inequality} to get $\left|\hat{N}_{m,n}\right| \leq 4096\sqrt{2}k^{-3/2} \left|b_{m,n}\right|$. This concludes the proof, as now
\begin{align*}
\Vert \mathcal{N}_k(0) \Vert=\sum_{m,n}\rho^{2(m+n+1)}\left|\hat{N}_{m,n}\right|\leq \sum_{m,n}\rho^{2(m+n+1)} 4096\sqrt{2}k^{-3/2} \left|b_{m,n}\right|=4 \left\Vert u_k^3\right\Vert
\end{align*}
and we can use Lemma \ref{lem:bounduk3} to get $\Vert \mathcal{N}_k(0)\Vert \leq 8\sqrt{2}k^{-3/2}$.
\end{proof}

\section[Operator A]{Operator $A$}\label{sec:A}
In general, there is no reason to expect that $\mathcal{N}_k$ is a local contraction, but one can assure it by choosing an appropriate form of the operator $A$. If $H_k$, the linear part of $\mathcal{N}_k$, satisfies $\Vert H_k\Vert<1$, then it is possible to find sufficiently small $\delta_k$ that $\mathcal{N}_k$ indeed is a contraction for $h_1,h_2\in B_{\delta_k}(0)$ (the question whether the image of this ball lies inside it is a separate matter). The operator $H_k$ can be written as
\begin{align*}
    H_k=I-\left( I+ 3 L_k^{-1} \circ \Lambda_k\right)\circ A,
\end{align*}
where we have introduced $\Lambda_k: X\to X$ defined as $\Lambda_k h=u_k^2 h$ (Lemma \ref{lem:triple_norm} ensures it is a well-defined, bounded operator). It leads us to the conclusion that $(I+3 L_k^{-1}\circ \Lambda_k)^{-1}$ would be the best choice for $A$. However, it is not feasible to calculate explicitly the form of such operator. Instead, we approximate it in a simple way described below. Let us stress on the fact that our choice of $A$ has been made with the line of proof of Lemma \ref{lem:boundHk} in mind and is in no way unique. 

Let us introduce
\begin{align*}
    \beta_0=& 4\sum_{j=0}^\infty f_j^2+2\sum_{j=0}^\infty f_j f_{j+1}+5f_0^2,\qquad
    \beta_1 =2\sum_{j=0}^\infty f_j f_{j+1}+3 f_0^2.
\end{align*}
\begin{obs}\label{obs:beta}
Using Lemma \ref{lem:q} we can establish
\begin{align*}
\frac{113}{1000}<\frac{1}{(1+q)^2}\left(4\frac{q}{1-q^2}+2\frac{q^2}{1-q^2}+5q\right) <\beta_0< 4\frac{q}{1-q^2}+2\frac{q^2}{1-q^2}+5\frac{q}{(1+q)^2}<\frac{135}{1000},
\end{align*}
\begin{align*}
\frac{38}{1000}<\frac{1}{(1+q)^2}\left(2\frac{q^2}{1-q^2}+3q\right) <\beta_1< \frac{2q^2}{1-q^2}+\frac{3q}{(1+q)^2}<\frac{45}{1000}.
\end{align*}
\end{obs}
\noindent These estimates allows us to define a linear operator $\mathcal{A}:Y_1\to Y_1$ that in the basis $(P_{0,0},P_{0,1},P_{1,0})$ is represented by the matrix
\begin{align*}
    \left[\begin{array}{ccc}-\frac{1}{24 \beta_0-1}&\frac{24\beta_1}{24 \beta_0-1}&\frac{24\beta_1}{24 \beta_0-1}\\
    0&1&0\\
    0&0&1\end{array}\right].
\end{align*}
Now let us denote by $\Pi$ a projection on the subspace $Y_1$. Then we define the linear operator $A:X\to X$ as
\begin{align*}
    A=\mathcal{A}\,\Pi+(I-\Pi).
\end{align*}
Its norm can be bounded from above in the following way.
\begin{lemma}\label{lem:boundA}
    It holds $\Vert A\Vert<\frac{139}{85}$.
\end{lemma}
\begin{proof}
Using Observation \ref{obs:beta} we get
\begin{align*}
\left\Vert AP_{0,0}\right\Vert=&\left|\frac{1}{24\beta_0-1}\right|\left\Vert P_{0,0}\right\Vert <\frac{10}{17}\left\Vert P_{0,0}\right\Vert,\\
\left\Vert AP_{0,1}\right\Vert=&\left|\frac{24\beta_1}{24\beta_0-1}\right|\left\Vert P_{0,0}\right\Vert + \left\Vert P_{0,1}\right\Vert <\frac{54}{85}\left\Vert P_{0,0}\right\Vert+\left\Vert P_{0,1}\right\Vert<\frac{139}{85}\left\Vert P_{0,1}\right\Vert,\\
\left\Vert AP_{1,0}\right\Vert=&\left|\frac{24\beta_1}{24\beta_0-1}\right|\left\Vert P_{0,0}\right\Vert + \left\Vert P_{1,0}\right\Vert <\frac{54}{85}\left\Vert P_{0,0}\right\Vert+\left\Vert P_{1,0}\right\Vert<\frac{139}{85}\left\Vert P_{1,0}\right\Vert.
\end{align*}
For the remaining $P_{m,n}$ the operator $A$ acts as an identity, hence, we get $\Vert A\Vert<\frac{139}{85}$.
\end{proof}

\section[Bound on the operator H]{Bound on the operator $H_k$}\label{sec:boundH}
Finally, we estimate the value of $\Vert H_k\Vert$. We start by the following result of simple but tedious calculations.

\begin{lemma}
For any $(m,n)\in\mathbb{N}^2$ one can decompose $u_k^2 P_{m,n}$ as
\begin{equation}\label{eq:u2Pmn}
    u_k^2 P_{m,n}=\sum_{\mu=-\infty}^\infty \sum_{\nu=-\infty}^\infty c_{\mu,\nu}P_{\mu+m,\nu+n},
\end{equation}
where
\begin{align*}
c_{\mu,\nu}=\frac{128}{k}
\begin{cases}
    \frac14 \sum_{j=0}^\infty f_j^2 & \mbox{for $\mu=\nu=0$,}\\
    \frac{1}{16}\left(\sum_{j=0}^\infty 2f_j f_{j+|\mu|} +\sum_{j=0}^{|\mu|-1} f_j f_{|\mu|-(j+1)} \right) & \mbox{for $|\mu|=|\nu|\neq 0$,}\\
    -\frac{1}{8} f^2_{\frac{|\mu|-1}{2}} & \mbox{for $\nu=0$ and $\mu\neq0$,}\\
    -\frac{1}{8} f^2_{\frac{|\nu|-1}{2}} & \mbox{for $\mu=0$ and $\nu\neq0$,}\\
    -\frac{1}{8} f_{\frac{|\mu|+|\nu|-1}{2}} f_{\frac{||\mu|-|\nu||-1}{2}} & \mbox{for $|\mu|,|\nu|\geq 1$, $|\mu|\neq |\nu|$, and $\mu-\nu$ odd,}\\
    0 & \mbox{for $|\mu|,|\nu|\geq 1$, $|\mu|\neq |\nu|$, and $\mu-\nu$ even.}\\
 \end{cases}
\end{align*}
\end{lemma}
\begin{proof}
We can write $u_k$ as 
\begin{align*}
    u_k=\sqrt{\frac{128}{k}}\sum_{l=0}^\infty f_l P_{l,l}=\sum_{\lambda=-\infty}^\infty \tilde{f}_{\lambda}P_{\lambda,\lambda},\quad\mbox{ where } \tilde{f}_\lambda = \sqrt{\frac{32}{k}}\begin{cases}
        f_{\lambda} &\mbox{for $\lambda\geq 0$}\\
        f_{-\lambda-1}& \mbox{for $\lambda< 0$}
    \end{cases}
\end{align*}
Then by decomposing $u_k^2P_{m,n}$ as
\begin{align*}
    u_k^2P_{m,n}=\sum_{\lambda=-\infty}^\infty \sum_{\kappa=-\infty}^\infty \tilde{f}_\lambda \tilde{f}_\kappa P_{\lambda,\lambda} P_{\kappa,\kappa} P_{m,n},
\end{align*}
using \eqref{eq:PPP} to write it as a linear combination of the basis functions, and collecting terms with the same indices in $P_{\mu,\nu}$ we obtain \eqref{eq:u2Pmn}.
\end{proof}

The coefficients $c_{\mu,\nu}$ can be bounded in the following way.
\begin{lemma}
    For $(\mu,\nu)\in\mathbb{Z}^2\setminus\{(0,0),(0,\pm1),(\pm1,0)\}$ it holds
    \begin{align}\label{eq:cest}
    \left|c_{\mu,\nu}\right|<\frac{1}{k}\left(8\max\{|\mu|,|\nu|\}+16\frac{q}{1-q^2}\right)q^{\max\{|\mu|,|\nu|\}}.
    \end{align}
\end{lemma}
\begin{proof}
Let us define
\begin{align*}
    g(n)=\frac{1}{k}\left(8n+16\frac{q}{1-q^2}\right)q^n.
\end{align*}
The result can be simply obtained with a case by case analysis using the estimation $f_n<q^{n+1/2}$. For $|\mu|=|\nu|$ we have
\begin{align*}
    \left|c_{\mu,\nu}\right|=\left|c_{|\mu|,|\nu|}\right|<\frac{8}{k}\left(\frac{2q^{|\mu|+1}}{1-q^2}+|\mu|\,q^{|\mu|}\right)=g(|\mu|).
\end{align*}
For $\nu=0$ and $|\mu|\geq 2$ there is
\begin{align*}
    \left|c_{\mu,\nu}\right|=\left|c_{|\mu|,0}\right|<\frac{16}{k}q^{|\mu|}\leq\frac{8|\mu|}{k}q^{|\mu|}<g(|\mu|).
\end{align*}
Analogous results holds in the opposite case, when $\mu=0$ and $|\nu|\geq 2$. Finally, for $|\mu|\neq|\nu|$, both larger than zero and with $\mu-\nu$ odd, at least one of the numbers $|\mu|$ and $|\nu|$ is no smaller than two, so
\begin{align*}
    \left|c_{\mu,\nu}\right|=\left|c_{|\mu|,|\nu|}\right|<\frac{16}{k}q^{\max\{|\mu|,|\nu|\}}\leq\frac{8}{k}\max\{|\mu|,|\nu|\}q^{|\mu|}<g\left(\max\{|\mu|,|\nu|\}\right).
\end{align*}
All these results can be gathered to give $\left|c_{|\mu|,|\nu|}\right|<g(\max\{|\mu|,|\nu|\})$.
\end{proof}

We estimate $\Vert H_k \Vert$ by showing a common bound on $\Vert H_k P_{m,n} \Vert$ for all $(m,n)\in\mathbb{N}^2$. 
It requires a control on the higher modes in the decomposition of $H_k P_{m,n}$. This is provided by the following lemma that is motivated by the observation.
\begin{obs}\label{obs:PPrho}
Since $\rho>1$ for any $(m,n)\in \mathbb{N}^2$ and $(\mu,\nu)\in \mathbb{Z}^2$ the triangle inequality gives us
\begin{align*}
    \frac{\left\Vert P_{m+\mu,n+\nu} \right\Vert}{\left\Vert P_{m,n} \right\Vert}=\rho^{|2m+2\mu+1|+|2n+2\nu+1|-(2m+1)-(2n+1)}<\rho^{|2\mu+1|+|2\nu+1|-2}<\rho^{2|\mu|+2|\nu|}<\rho^{4\max\{|\mu|,|\nu|\}}.
\end{align*}    
\end{obs}
\begin{lemma}\label{lem:estJL}
Let us fix $l\in \mathbb{N}_+$ and define 
\begin{align*}
J_l=\left\{(\mu,\nu)\in\mathbb{Z}^2:|\mu|> l \mbox{ or } |\nu|> l\right\}.
\end{align*}
Then
\begin{align}\label{eq:estJL}
\sum_{(\mu,\nu)\in J_l} |c_{\mu,\nu}|\rho^{4\max\{|\mu|,|\nu|\}}<\frac{64 \left(q \rho ^4\right)^{l+1} (\alpha_2l^2+\alpha_1 l+\alpha_0)}{k \left(1-q^2\right) \left(1-q \rho ^4\right)^3},
\end{align}
where
\begin{align*}
\alpha_0&=\left(1+2q-q^2\right)+\left(1-2q-q^2\right)\rho^4 q,\\
\alpha_1&=2\left(1+q-q^2\right)-2\left(1+2q-q^2\right)\rho^4 q+2q(\rho^4 q)^2,\\
\alpha_2&=\left(1-q^2\right)\left(1-q\rho^4\right)^2.
\end{align*}
\end{lemma}

\begin{proof}
Since $l\geq 1$, we can use \eqref{eq:cest} to estimate
\begin{align*}
\sum_{(\mu,\nu)\in J_l} |c_{\mu,\nu}|\rho^{4\max\{|\mu|,|\nu|\}}\leq\sum_{(\mu,\nu)\in J_l}\frac{1}{k} \left(8\max\{|\mu|,|\nu|\}+16\frac{q}{1-q^2}\right) (\rho^4 q)^{\max\{|\mu|,|\nu|\}}=:\sum_{(\mu,\nu)\in J_l}d_{\max\{|\mu|,|\nu|\}},
\end{align*}
where we have denoted the terms inside the sum by $d_{\max\{|\mu|,|\nu|\}}$. Then we can split this sum into
\begin{align*}
\sum_{(\mu,\nu)\in J_l} d_{\max\{|\mu|,|\nu|\}}&=
2\sum_{\mu=l+1}^\infty \sum_{\nu=-\infty}^\infty d_{\max\{|\mu|,|\nu|\}}
+ 2\sum_{\mu=-l}^l \sum_{\nu=l+1}^\infty d_{\max\{|\mu|,|\nu|\}}\\
&=2\sum_{\mu=l+1}^\infty d_\mu+4 \sum_{\mu=l+1}^{\infty} \sum_{\nu=1}^{\mu} d_\mu +4 \sum_{\mu=l+1}^{\infty} \sum_{\nu=\mu+1}^\infty d_\nu+ 2\sum_{\mu=-l}^l \sum_{\nu=l+1}^\infty d_\nu\\
&=2\sum_{\mu=l+1}^\infty d_\mu+4 \sum_{\mu=l+1}^{\infty} \mu\, d_\mu +4 \sum_{\nu=l+2}^\infty (\nu-l-1) d_\nu+ 2(2l+1) \sum_{\nu=l+1}^\infty d_\nu\\
&=8 \sum_{\mu=l+1}^{\infty} \mu\, d_\mu.
\end{align*}
Using the identities
\begin{align*}
\sum_{m=l+1}^\infty p^m = \frac{p^{l+1}}{1-p},\quad \sum_{m=l+1}^\infty m\,p^m = \frac{p^{l+1}[1+l(1-p)]}{(1-p)^2},\quad \mbox{ for $|p|<1$},
\end{align*}
we end up with \eqref{eq:estJL}.
\end{proof}

\begin{obs}\label{obs:alpha}
Since $q\in(\frac{13}{1000},\frac{15}{1000})$ and $\rho=1+\frac{1}{1000}$, we can bound
\begin{align*}
    \alpha_0<\frac{1045}{1000},\quad \alpha_1<\frac{2003}{1000},\quad \alpha_2<\frac{974}{1000}.
\end{align*}
\end{obs}

For any fixed $(m,n)\in\mathbb{N}^2$ let us introduce the decomposition
\begin{align*}
    L_k^{-1}\left(u_k^2 P_{m,n}\right)=\mathcal{I}_k^{(m,n)}+\mathcal{J}_k^{(m,n)},
\end{align*}
where
\begin{align*}
    \mathcal{I}_k^{(m,n)}= \sum_{\mu=-1}^{1} \sum_{\nu=-1}^{1} c_{\mu,\nu} L^{-1}P_{\mu+m,\nu+n},\qquad
    \mathcal{J}_k^{(m,n)}= \sum_{(\mu,\nu)\in J_1} c_{\mu,\nu} L^{-1}P_{\mu+m,\nu+n}.
\end{align*}
The first part can be written explicitly as
\begin{align}\label{eq:I1}
     \mathcal{I}_k^{(m,n)}=&\frac{8}{k}\left(\sum_{j=0}^\infty 2f_j f_{j+1} +f_0^2 \right)\left[L_k^{-1}P_{m-1,n-1}+L_k^{-1}P_{m-1,n+1}+L_k^{-1}P_{m+1,n-1}+L_k^{-1}P_{m+1,n+1}\right]\\
    &-\frac{16}{k}f_0^2\left[L_k^{-1}P_{m,n-1}+L_k^{-1}P_{m,n+1}+L_k^{-1}P_{m-1,n}+L_k^{-1}P_{m+1,n}\right]+\frac{32}{k}\left(\sum_{j=0}^\infty f_j^2\right) L_k^{-1}P_{m,n}.\nonumber
\end{align}
The second part is a subject to the universal bound.
\begin{lemma}\label{lem:boundI2}
For any $(m,n)\in \mathbb{N}^2$ and $k\geq100$ it holds $\left\Vert \mathcal{J}_k^{(m,n)}\right\Vert < \frac{1}{16} \Vert P_{m,n}\Vert$.
\end{lemma}
\begin{proof}
Using \eqref{eq:L-1negative}, Observation \ref{obs:PPrho}, and Lemma \ref{lem:estJL} we have the following inequalities
\begin{align*}
\left\Vert \mathcal{J}_k^{(m,n)}\right\Vert&\leq\sum_{(\mu,\nu)\in J_1} \left|c_{\mu,\nu}\right|\left\Vert L^{-1}P_{\mu+m,\nu+n}\right\Vert\leq \frac{4k^2}{4k-1} \left\Vert P_{m,n}\right\Vert \sum_{(\mu,\nu)\in J_1} \left|c_{\mu,\nu}\right|\rho^{4\max\{|\mu|,|\nu|\}}\\
&\leq \frac{256 k\, q^2 \rho^8 (\alpha_2+\alpha_1+\alpha_0)}{(4k-1) \left(1-q^2\right) \left(1-q \rho ^4\right)^3}\left\Vert P_{m,n}\right\Vert.
\end{align*}
The last expression is a decreasing function of $k$, thus, we can get the bound by fixing $k=100$ and applying Lemma \ref{lem:q} together with Observation \ref{obs:alpha}.
\end{proof}

Now, having \eqref{eq:I1} and Lemma \ref{lem:boundI2}, we are ready to prove the main result of this section.

\begin{lemma}\label{lem:boundHk}
For $k\geq 100$ it holds $\left\Vert H_k\right\Vert < \frac{88}{100}$.
\end{lemma}
\begin{proof}
We estimate $\Vert H_k P_{m,n}\Vert$ in different ways depending on the indices $(m,n)$. Let us start with $\Vert H_k P_{0,0}\Vert$, we have
\begin{align}\label{eq:HkP00}
    H_k P_{0,0}=\frac{3}{24\beta_0-1}\mathcal{I}_k^{(0,0)} +\frac{24\beta_0}{24\beta_0 -1}P_{0,0} + \frac{3}{24\beta_0-1}\mathcal{J}_k^{(0,0)}.
\end{align}
The first two terms in this expression can be written explicitly as
\begin{align*}
 \frac{3}{24\beta_0-1}\mathcal{I}_k^{(0,0)} +\frac{24\beta_0}{24\beta_0 -1}P_{0,0}=
\frac{24\beta_0}{(24\beta_0 -1)(4k+1)}P_{0,0}
- \frac{96\beta_1 k}{(24\beta_0 -1)(32k^2-4k-1)}P_{0,1}\\
+\frac{96\beta_1 k}{(24\beta_0 -1)(32k^2+36k+9)}P_{1,0}
-\frac{96\left(\beta_1-2f_0^2\right)k}{9(24\beta_0 -1)(4k+1)}P_{1,1}.
\end{align*}
Estimates from Lemma \ref{lem:q} and Observation \ref{obs:beta} let us bound the norm of this expression by $\frac{4}{100}\rho^2$, where we have used the fact that fractions present in the first three terms are decreasing functions of $k$, so they can be constrained by the evaluation at $k=100$. For the fourth term we just use $\frac{k}{4k+1}<\frac{1}{4}$. The last part in \eqref{eq:HkP00} can be bounded with Lemma \ref{lem:boundI2}. Together, we get $\Vert H_k P_{0,0} \Vert< \frac{15}{100}\Vert P_{0,0}\Vert$.

We proceed similarly for $\Vert H_k P_{0,1}\Vert$. This time we have the following decomposition
\begin{align*}
    H_k P_{0,1}=-\frac{72\beta_1}{24\beta_0-1}\mathcal{I}_k^{(0,0)}  -\frac{72\beta_1}{24\beta_0-1}\mathcal{J}_k^{(0,0)} -3 \mathcal{I}_k^{(0,1)} -3 \mathcal{J}_k^{(0,1)} -\frac{24\beta_1}{24\beta_0 -1}P_{0,0}.
\end{align*}
Its finite part can be written as
\begin{align*}
-\frac{72\beta_1}{24\beta_0-1}\mathcal{I}_k^{(0,0)}   -3 \mathcal{I}_k^{(0,1)} -\frac{24\beta_1}{24\beta_0 -1}P_{0,0}=\frac{24 \beta_1}{(24 \beta_0-1) (4 k+1)}P_{0,0}\\
-192 k \frac{24 \left(\beta_0^2-2 \beta_1^2\right)-(24 \beta_1+1) (\beta_0-\beta_1)}{2 (24 \beta_0-1) \left(32 k^2-4 k-1\right)}P_{0,1}
+96 k\frac{(24 \beta_0-1) \left(\beta_1-2 f_0^2\right)-24 \beta_1^2}{(24 \beta_0-1) \left(32 k^2+36 k+9\right)}P_{1,0}\\
+ 64k\frac{f_0^2 (24 (\beta_0+ \beta_1)-1)-12 \beta_1^2}{3 (24 \beta_0-1) (4 k+1)}P_{1,1}+\frac{96 \beta_1 k}{96 k^2-4 k-1}P_{0,2}
-\frac{96 k \left(\beta_1-2 f_0^2\right)}{64 k^2-36 k-9}P_{1,2}
\end{align*}
and its norm is smaller than $\frac{11}{100}\rho^4$. Together with Lemma \ref{lem:boundI2} it gives $\Vert H_k P_{0,1} \Vert< \frac{3}{10}\Vert P_{0,1}\Vert$. With the same approach we write
\begin{align*}
    H_k P_{1,0}=-\frac{72\beta_1}{24\beta_0-1}\mathcal{I}_k^{(0,0)}  -\frac{72\beta_1}{24\beta_0-1}\mathcal{J}_k^{(0,0)} -3 \mathcal{I}_k^{(1,0)} -3 \mathcal{J}_k^{(1,0)} -\frac{24\beta_1}{24\beta_0 -1}P_{0,0}.
\end{align*}
This time the finite part, being
\begin{align*}
-\frac{72\beta_1}{24\beta_0-1}\mathcal{I}_k^{(0,0)}   -3 \mathcal{I}_k^{(1,0)} -\frac{24\beta_1}{24\beta_0 -1}P_{0,0}=
-\frac{24 \beta_1}{(24 \beta_0-1) (4 k+1)}P_{0,0}\\
-96 k \frac{(24 \beta_0-1) \left(\beta_1-2 f_0^2\right)-24 \beta_1^2}{(24 \beta_0-1) \left(32 k^2-4 k-1\right)}P_{0,1}
+96k\frac{24 \left(\beta_0^2-\beta_0 \beta_1-\beta_1^2\right)-\beta_0+\beta_1}{(24 \beta_0-1) \left(32 k^2+36 k+9\right)}P_{1,0}\\
-64k\frac{-12 \beta_1^2+(24 ( \beta_0+\beta_1)-1) f_0^2}{3 (24 \beta_0-1) (4 k+1)}P_{1,1}
-\frac{96 \beta_1 k}{96 k^2+100 k+25} P_{2,0}+96 k \frac{\beta_1-2 f_0^2}{64 k^2+100 k+25}P_{2,1},
\end{align*}
leads to the overall bound $\Vert H_k P_{1,0} \Vert< \frac{3}{10}\Vert P_{1,0}\Vert$.

Now let us consider $P_{m,n}\in Y_2\setminus Y_1$. Here we have thirteen cases, but for all of them the operator $A$ is just an identity. Hence, we have
\begin{align}\label{eq:HkPmn}
    H_k P_{m,n} = -3\mathcal{I}_k^{(m,n)} -3\mathcal{J}_k^{(m,n)}.
\end{align}
The second part is again treated by Lemma \ref{lem:boundI2}. The first part consists of nine terms as presented in \eqref{eq:I1} that can be bounded separately. Due to the presence of the operator $L_k^{-1}$, each of them depends on $k$ as either $k/(4k+1)$ or $k/(\gamma_2 k^2+\gamma_1 k+\gamma_0)$, where $\gamma_2>0$. In the first case, it can be easily bounded from above using $\frac{k}{4k+1}<\frac{1}{4}$, while in the second case, we can just evaluate it at $k=100$ as it is a decreasing function. In the end we get for $k\geq 100$ the following bounds
\begin{align*}
&&  && \left\Vert H_k P_{0,2} \right \Vert<\frac{24}{100}, && \left\Vert H_k P_{0,3} \right \Vert<\frac{19}{100},\\
&& \left\Vert H_k P_{1,1} \right \Vert<\frac{74}{100}, && \left\Vert H_k P_{1,2} \right \Vert<\frac{30}{100}, && \left\Vert H_k P_{1,3} \right \Vert<\frac{21}{100}, \\
\left\Vert H_k P_{2,0} \right \Vert<\frac{24}{100}, && \left\Vert H_k P_{2,1} \right \Vert<\frac{30}{100},&& \left\Vert H_k P_{2,2} \right \Vert<\frac{30}{100}, && \left\Vert H_k P_{2,3} \right \Vert<\frac{24}{100},\\
\left\Vert H_k P_{3,0} \right \Vert<\frac{19}{100}, && \left\Vert H_k P_{3,1} \right \Vert<\frac{21}{100}, && \left\Vert H_k P_{3,2} \right \Vert<\frac{24}{100}, && \left\Vert H_k P_{3,3} \right \Vert<\frac{24}{100}.
\end{align*}

Finally, let us consider the case when $P_{m,n}\in Z_2$. As $A|_{Z_2}=I$, the decomposition \eqref{eq:HkPmn} still holds. This time, we use \eqref{eq:L-1negative2} and an estimate $f_n<q^{n+1/2}$ to get
\begin{align*}
    \left\Vert \mathcal{I}_k^{(m,n)}\right\Vert \leq&\left[\left(q+\frac{2q^2}{1-q^2}\right)\left(\frac{32k}{4 k\max\{|2m-1|,|2n-1|\}-1}+\frac{32k}{4 k\max\{|2m-1|,2n+3\}-1}\right.\right.\\
    &+\frac{32k}{4 k\max\{2m+3,|2n-1|\}-1}+\left.\frac{32k}{4 k\max\{2m+3,2n+3\}-1}\right) \\
    &+q\left(\frac{64k}{4 k\max\{|2m-1|,2n+1\}-1}+\frac{64k}{4 k\max\{2m+1,|2n-1|\}-1}\right.\\
    &\left. +\frac{64k}{4 k\max\{2m+1,2n+3\}-1}+\frac{64k}{4 k\max\{2m+3,2n+1\}-1}\right)\\
    &\left.+\frac{q}{1-q^2}\frac{128k}{4 k\max\{2m+1,2n+1\}-1}\right]\rho^{2m+2n+6}.
\end{align*}
Since $P_{m,n}\in Z_2$, either $m\geq 4$ or $n\geq 4$ and $\max$ functions can be bounded leading to
\begin{align*}
    \left\Vert \mathcal{I}_k^{(m,n)}\right\Vert  \leq&\left[\left(q+\frac{q^2}{1-q^2}\right)\frac{128k}{4 k\cdot 7-1}+\left(q+\frac{q}{1-q^2}\right)\frac{128k}{4 k\cdot 9-1}+\left(q+\frac{q^2}{1-q^2}\right)\frac{128k}{4 k\cdot 11-1}\right]\rho^{4}\left\Vert P_{m,n}\right\Vert.
\end{align*}
The expression on the right hand side is decreasing in $k$, so by plugging $k=100$ we obtain the bound $\Vert\mathcal{I}_k^{(m,n)}\Vert<\frac{23}{100}\left\Vert P_{m,n}\right\Vert$. Taking into account Lemma \ref{lem:boundI2} we conclude that in total $\Vert H_k P_{m,n}\Vert<\frac{88}{100}\Vert P_{m,n}\Vert$.

Since for every $k\geq100$ and $(m,n)\in\mathbb{N}^2$ it holds $\Vert H_k P_{m,n}\Vert<\frac{88}{100}$, it concludes the proof.
\end{proof}

\printbibliography

\end{document}